\documentclass{amsproc}

\def\cal{\mathcal}

\def\C{{\cal C}}

\def\V{{\cal V}}

\def\frak{\mathfrak}

\def\da{{\frak a}}
\def\db{{\frak b}}

\def\dg{{\frak g}}
\def\dgl{\dg\dl}
\def\dh{{\frak h}}

\def\dl{{\frak l}}

\def\dgo{{\frak o}}
\def\dgp{{\frak p}}

\def\ds{{\frak s}}
\def\dsl{\ds\dl}
\def\dsp{\ds\dgp}

\def\du{{\frak u}}
\def\dU{{\frak U}}

\def\osp{\dgo\ds\dgp}

\def\Bbb{\mathbb}

\def\bC{\Bbb C}

\def\bN{\Bbb N}

\def\bP{\Bbb P}

\def\bR{\Bbb R}

\def\bZ{\Bbb Z}

\def\ep{\epsilon}

\def\la{\langle}

\def\lf{\lfloor}
\def\o{\overline}
\def\ot{\otimes}

\def\ra{\rangle}

\def\rf{\rfloor}

\def\t{\tilde}

\def\and{\mbox{\rm \ and\ }}

\def\Bol{\mathop{\rm Bol}\nolimits}

\def\Comp{\mathop{\rm Comp}\nolimits}
\def\Con{\mathop{\rm Con}\nolimits}
\def\CQ{\mathop{\rm CQ}\nolimits}

\def\Diff{\mathop{\rm Diff}\nolimits}

\def\degree{\mathop{\rm degree}\nolimits}

\def\Div{\mathop{\rm Div}\nolimits}

\def\End{\mathop{\rm End}\nolimits}

\def\Ext{\mathop{\rm Ext}\nolimits}

\def\fine{{\mathop{\rm fine}\nolimits}}

\def\Hom{\mathop{\rm Hom}\nolimits}

\def\Ind{\mathop{\rm Ind}\nolimits}

\def\lac{\mathop{\rm lac}\nolimits}

\def\lp{{\lambda,p}}

\def\ndup{{\mbox{\scrm nd}}}

\def\odd{{\mathop{\rm odd}\nolimits}}

\def\oo{{1|1}}

\def\Poly{\mathop{\rm Poly}\nolimits}
\def\PQ{\mathop{\rm PQ}\nolimits}

\def\px{\partial_x}
\def\pxi{\partial_\xi}

\def\Roo{{\bR^{1|1}}}

\def\SBol{\mathop{\rm SBol}\nolimits}

\def\scrm{\scriptsize\rm}

\def\Span{\mathop{\rm Span}\nolimits}
\def\SQ{\mathop{\rm SQ}\nolimits}

\def\Stabilizer{\mathop{\rm Stabilizer}\nolimits}

\def\stup{{\mbox{\scrm st}}}

\def\Symb{\mathop{\rm Symb}\nolimits}

\def\Tan{\mathop{\rm Tan}\nolimits}
\def\thup{{\mbox{\scrm th}}}

\def\ts{\textstyle}

\def\Vec{\mathop{\rm Vec}\nolimits}

\def\VR{\Vec \Bbb R}

\def\VRm{\Vec\Bbb R^m}
\def\VRoo{\Vec\Roo}

\def\cind{completely indecomposable}

\def\dog{differential operator}
\def\eg{{\em e.g.,\/}}

\def\ie{{\em i.e.,\/}}
\def\iff{if and only if}

\def\ind{indecomposable}

\def\irr{irreducible}

\def\lsa{Lie superalgebra}
\def\lwv{lowest weight vector}

\def\psdog{pseudo\dog}
\def\psidog{$\Psi$DO}

\def\r{representation}

\def\sq{subquotient}

\def\tdm{tensor density module}

\def\tfm{tensor field module}
\def\th{\thinspace}

\def\uea{universal enveloping algebra}

\def\vf{vector field}
\def\vs1{$\Vec(S^1)$}

\def\Symbf{\Symb_\fine}

\newtheorem{theorem}{Theorem}[section]
\newtheorem{lemma}[theorem]{Lemma}
\newtheorem{prop}[theorem]{Proposition}

\theoremstyle{definition}
\newtheorem{definition}[theorem]{Definition}
\theoremstyle{remark}

\numberwithin{equation}{section}

\begin{document}

\title{Quantizations of modules of differential operators}

\author{Charles H.\ Conley}
\address{Department of Mathematics, University of North Texas, Denton, Texas 76203, USA}
\email{conley@unt.edu}


\dedicatory{This article is dedicated with admiration and affection to my advisor, \\Professor V.~S.~Varadarajan, on the occasion of his 70$^{\mbox{\scriptsize\it th}}$ birthday}



\maketitle

\section*{Introduction}

Fix a manifold $M$, and let $\V$ be an infinite dimensional simple Lie subalgebra of the Lie algebra $\Vec M$ of vector fields on $M$.  Assume that $\V$ contains a finite dimensional simple maximal subalgebra $\da(\V)$.  We define an {\em $\da(\V)$-quantization\/} of a $\V$-module of differential operators on $M$ to be a decomposition of the module into irreducible $\da(\V)$-modules.  In this article we survey some recent results and open problems involving this type of quantization and its applications to cohomology, \ind\ modules, and geometric equivalences and symmetries of differential operator modules.

There are several mathematical theories of quantization.  Two of the most important are geometric quantization, which hinges on polarization and is linked to the orbit method in the \r\ theory of Lie groups, and deformation quantization, in which the classical Poisson algebra structure becomes the first order approximation of an associative star product.

In its original physical sense, to quantize a system meant to replace the commutative Poisson algebra of functions on the phase space, the classical observables, with a noncommutative algebra of operators on a Hilbert space, the quantum mechanical observables.  In the theory of quantization under consideration here, the role of the noncommutative algebra is played by the differential operators and that of the commutative algebra is played by their symbols.

We will consider two cases: the case that $\V$ is all of $\Vec M$, and the case that $M$ is a contact manifold and $\V$ is the Lie algebra $\Con M$ of contact vector fields on $M$.  Our approach is algebraic: we assume that $M$ is a Euclidean manifold $\bR^m$ and we consider only polynomial vector fields.  Thus, writing $D_i$ for $\partial/\partial x_i$ and using the multi-index notation $x^J = x_1^{J_1} \cdots x_m^{J_m}$,
$$ \V \subseteq \VRm := \Span_\bC \bigl\{ x^J D_i: 1\le i\le m, J\in\bN^m \bigr\}. $$

\section{Projective quantizations} \label{PQs}

In the case that $\V$ is all of $\VRm$, we take $\da(\V)$ to be the {\em projective algebra\/} $\da_m$, a copy of $\dsl_{m+1}$.  Writing $E$ for the Euler operator $\sum_1^m x_iD_i$,
$$ \da_m := \Span_\bC \bigl\{ D_i, x_jD_i, x_jE: 1\le i,j \le m \bigr\} \cong \dsl_{m+1}. $$

The $\da_m$-quantizations are called {\it projective quantizations.\/}  The first example is the projective quantization of the associative algebra $\Diff\bR^m$ of polynomial \dog s on $\bR^m$.  Denoting $D_1^{I_1}\cdots D_m^{I_m}$ by $D^I$,
$$ \Diff\bR^m := \Span_\bC \bigl\{ x^JD^I: I,J\in\bN^m \bigr\}. $$

Let us write $\sigma$ for the two-sided action of $\VRm$ on $\Diff\bR^m$.  It is a derivation action which preserves the {\it order filtration\/} $\Diff^k\bR^m$.  The associated \sq s are the {\it symbol modules:}
$$ \Symb^k\bR^m := \Diff^k\bR^m \big/ \Diff^{k-1}\bR^m. $$
Write $\sigma_k$ for the action of $\VRm$ on $\Symb^k\bR^m$, and let $\Symb\bR^m$ be the {\em total symbol module\/} $\bigoplus_{k=0}^\infty \Symb^k \bR^m$, the graded algebra of $\Diff\bR^m$.

\begin{prop} \label{PQ}
There exists a unique $\da_m$-equivalence
$$ \PQ: \Symb\bR^m \to \Diff\bR^m, $$
the\/ {\em projective quantization,\/} which is the identity on symbols.
\end{prop}

\begin{proof}
By ``the identity on symbols'', we mean that for all $S\in\Symb^k\bR^m$, $\PQ(S)$ is in $\Diff^k\bR^m$ and has symbol $S$.  It is not hard to check that the symbol modules are duals of relative Verma modules of $\da_m$ with distinct infinitesimal characters (indeed, distinct Casimir eigenvalues), whence the result.
\end{proof}

The explicit formula for $\PQ$ was obtained independently by Cohen, Manin, and Zagier (for $m=1$) \cite{CMZ97}, and by Lecomte and Ovsienko (for all $m$) \cite{LO99}.  Our theme in this article is the action $\sigma$ of $\VRm$ on $\Diff\bR^m$ ``in terms of $\PQ$'', by which we mean the action on $\Symb\bR^m$ obtained by transferring $\sigma$ via $\PQ$.  As we will see, the explicit formula for the transferred action $\pi$ contains geometric and cohomological information.  We now define $\pi$ and give a lemma stating its most elementary properties.

\begin{definition} \label{pi}
Let $\pi$ be the action\/ $\PQ^{-1} \circ \sigma \circ \PQ$ of\/ $\VRm$ on\/ $\Symb\bR^m$.  Regard it as an infinite matrix with entries
$$ \pi_{ij}: \VRm \to \Hom_\bC(\Symb^j\bR^m,\Symb^i\bR^m),\ \ i,j\in\bN. $$
\end{definition}

\begin{lemma} \label{piij}
\begin{enumerate}
\item[(a)]  The matrix $\pi$ is upper triangular.
\item[(b)]  Its diagonal entries are $\pi_{ii} = \sigma_i$, the actions on the\/ $\Symb^i\bR^m$.
\item[(c)]  For $i<j$, $\pi_{ij}$ is $\da_m$-covariant and zero on $\da_m$.
\end{enumerate}
\end{lemma}

\begin{proof}
Part~(a) is due to the fact that $\sigma$ and $\PQ$ preserve the filtration $\Diff^k\bR^m$.  Part~(b) holds because $\PQ$ is the identity on symbols.  Part~(c) follows from the $\da_m$-covariance of $\PQ$.
\end{proof}

One of our central goals is to compute the matrix entries $\pi_{ij}$.  As a \r\ of $\da_m$, $\VRm/\da_m$ is an irreducible lowest weight module, so by Lemma~\ref{piij}c each $\pi_{ij}$ is determined by its value on the \lwv.  This \lwv\ is $x_1^3 D_1$ for $m=1$ and $x_m^2 D_1$ for $m > 1$, explaining why we will see a dichotomy between these two cases.

Let us give two examples of the kind of data the $\pi_{ij}$ contain.  First, the \sq\ $\Diff^k\bR^m/\Diff^l\bR^m$ splits as $\bigoplus_{l<i\le k}\Symb^i\bR^m$ under $\VRm$ \iff\ $\pi_{ij}=0$ for $l< i<j \le k$.  Such splittings are of geometric interest.

Second, Lemma~\ref{piij}c says that the upper triangular entries are {\em $\da_m$-relative 1-cochains.\/}  The fact that $\pi$ is a \r\ translates to the {\it cup equation:\/}
\begin{equation} \label{cup}
   \partial \pi_{ij} + \sum_{i<r<j} \pi_{ir}\cup\pi_{rj} \ =\ 0,
\end{equation}
where $\partial$ is the coboundary operator.  In particular, the entries $\pi_{i,i+1}$ on the first superdiagonal are 1-cocycles.  The uniqueness of $\PQ$ implies that they are cohomologically trivial \iff\ they are zero, so the non-zero entries are a source of non-trivial cohomology classes.

\subsection*{Tensor field modules}
In computing the $\pi_{ij}$, one is led to a general class of projective quantizations.  Observe that the symbol modules are algebraic sections of completely reducible vector bundles over $\bR^m$ of finite dimensional fiber.  Such $\VRm$-modules are {\em \tfm s.\/}  Other examples are the alternating forms and the tensor densities.  In fact, all \tfm s arise as sections of subbundles of tensor products of the form and density bundles.

Given any two \tfm s $F$ and $E$, we have the $\VRm$-module $\Diff(F,E)$ of \dog s from $F$ to $E$.  It is filtered by order, and the associated symbol modules $\Symb^k(F,E)$ are again \tfm s.

It is not hard to see that the matrix entries $\pi_{ij}$ defined earlier are \dog-valued.  Since they are $\da_m$-covariant, it becomes necessary to understand the decomposition of $\Diff(\Symb^j\bR^m, \Symb^i\bR^m)$ under $\da_m$.  In other words, we must study the projective quantizations not only of ordinary \dog s, but also of \dog s between symbol modules, the study of which leads to still other quantizations.  The appropriate level of generality is attained by studying the projective quantizations $\PQ_{F,E}$ of all modules $\Diff(F, E)$, where $F$ and $E$ are arbitrary \tfm s.

Tensor field modules are in bijection with completely reducible finite dimensional \r s of $\dgl_m$.  To explain, we must define certain Lie subalgebras of $\VRm$.  Let $\Vec_n\bR^m$ be the algebra of vector fields vanishing to order at least~$n+1$ at the origin, and let $\du_m$ and $\dl_m$ be the {\em constant\/} and {\em linear} algebras, respectively:
\begin{displaymath} \begin{array}{rclcl}
   \du_m &:=& \Span_\bC \bigl\{ D_i : 1\le i\le m \bigr\} &\cong& \bC^m, \\[4pt]
   \dl_m &:=& \Span_\bC \bigl\{ x_jD_i : 1\le i,j\le m \bigr\} &\cong& \dgl_m.
\end{array} \end{displaymath}
Clearly $\VRm = \du_m \oplus \Vec_0\bR^m$ and $\Vec_0\bR^m = \dl_m \oplus \Vec_1\bR^m$.  Furthermore, $\Vec_1\bR^m$ is an ideal in $\Vec_0\bR^m$.  For reference, let $\db_m$ be the {\em affine algebra\/} $\du_m \oplus \dl_m$.

Given any $\dl_m$-module $V$, define a $\VRm$-module $F(V)$ by
$$ F(V) := \bigl( \Ind_{\dU(\Vec_0\bR^m)}^{\dU(\VRm)} V^* \bigr)^* .$$
(Here $\dU$ denotes the \uea, $V$ is extended trivially to $\Vec_0\bR^m$, and the outer dual is restricted so that $\dl_m$ acts locally finitely.)  Then $F(V)$ is the module of sections of the bundle with fiber $V$, and $V \mapsto F(V)$ is the bijection from completely reducible finite dimensional $\dl_m$-modules to \tfm s.  Note that the $\du_m$-invariant subspace $F(V)^{\du_m}$ of $F(V)$, the ``lowest $\dl_m$-module'' of $F(V)$, is $V$ itself.  Hence the inverse of $V \mapsto F(V)$ is $F \mapsto F^{\du_m}$.  

In order to describe the symbol modules of $\Diff(F,E)$, we recall the theory of \irr\ finite dimensional \r s of $\dl_m$.  Let $\dh_m$ be the Cartan subalgebra $\Span_\bC\{x_iD_i: 1\le i\le m\}$.  Given $\lambda\in\bC^m$, let $L(\lambda)$ be the \irr\ $\dl_m$-module of lowest weight $\lambda$, by which we mean that $x_iD_i$ acts on the \lwv\ by $\lambda_i$ and $x_jD_i$ annihilates it for all $i<j$.  The finite dimensional $\dl_m$-modules are precisely those $L(\lambda)$ such that $\lambda_i-\lambda_{i-1}$ is in $\bN$ for all~$i$.  The dual $L(\lambda)^*$ is $L(\lambda^*)$, where $\lambda^*$ is defined to be $(-\lambda_m,\ldots,-\lambda_1)$.

The following examples are useful.  The space of homogeneous polynomials of degree~$j$ is $\dl_m$-invariant and has \lwv\ $x_m^j$, which has weight $je_m$ (we write $e_i$ for the $i^\thup$ standard basis vector of $\bC^m$).  Therefore it is $L(je_m)$, the $j^\thup$ symmetric power of $L(e_m)$.  The $\du_m$-invariant subspace of $\Symb^k \bR^m$ is the span of the constant symbols $\{D^I: |I|=k\}$, which has \lwv\ $D_1^k$ and is $L(-ke_1)$, the dual of $L(ke_m)$.

As we stated, for any $\dl_m$-modules $V$ and $W$, $\Symb^k \bigl(F(V), F(W)\bigr)$ is itself a \tfm.  Its $\du_m$-invariant $\dl_m$-submodule is $L(-ke_1) \ot V^* \ot W$, so
\begin{equation} \label{VWsymbol}
   \Symb^k \bigl(F(V), F(W)\bigr) = F \bigl(L(-ke_1) \ot V^* \ot W\bigr).
\end{equation}

As an $\da_m$-module, $F(V)$ is the dual of the $\dl_m$-relative Verma module induced by $V^*$.  For generic choices of $V$ and $W$, no two of the symbol modules (\ref{VWsymbol}) have any $\da_m$-infinitesimal characters in common.  In these cases the projective quantization $\PQ_{F(V), F(W)}$ exists: it is the unique symbol-preserving $\da_m$-equivalence
\begin{equation} \label{genPQ}
   \PQ_{F(V), F(W)}: \bigoplus_{k=0}^\infty F \bigl(L(-ke_1) \ot V^* \ot W\bigr)
   \to \Diff \bigl(F(V), F(W)\bigr).
\end{equation}

\subsubsection*{The resonant case}
Those choices of $V$ and $W$ for which the symbol modules share $\da_m$-infinitesimal characters are called {\em resonant.\/}  These cases are singular, but nevertheless play a crucial role even in the study of the non-resonant cases.  Usually the resonant cases do not admit projective quantizations, and when they do, the quantizations are not unique without further restrictions.

\subsection*{Research problems}

We will be guided by the following five problems, which are not fully solved and offer directions for research.  As we will see, they are tightly related, and Problems~2 and~3 in some sense govern the others.  We only formulate them for \dog s, but they make sense for \psdog s.

Let $F$ and $E$ be arbitrary \tfm s.

\medbreak{\sc Problem~1.}
Describe the action of $\VRm$ on $\Diff(F,E)$ in terms of $\PQ_{F,E}$.

\medbreak{\sc Problem~2.}
Describe composition in terms of $\PQ_{F,E}$.

\medbreak{\sc Problem~3.}
Compute the cohomology of $F$ and $\Diff(F,E)$.

\medbreak{\sc Problem~4.}
Which \sq s of the $\Diff(F,E)$ are equivalent?

\medbreak{\sc Problem~5.}
Classify uniserial extensions of \tfm s.

\subsection*{Tensor density modules}
For $\lambda\in\bC$, let $\bC_\lambda$ denote the 1-dimensional module of $\dl_m$ in which the Euler operator $E$ acts by $m\lambda$, the module $L(\lambda, \ldots, \lambda)$.  The $\VRm$-modules $F(\bC_\lambda)$ are the {\em \tdm s,\/} the simplest \tfm s.  We will write $F(\lambda)$ for $F(\bC_\lambda)$, and $\pi_\lambda$ for the action of $\VRm$ on it.  This module may be expressed concretely as follows:
\begin{equation*}
   F(\lambda) := dx^\lambda \bC[x_1, \ldots, x_n], \quad
   \pi_\lambda(X)(dx^\lambda f) := dx^\lambda \bigl(X(f) + \lambda f \nabla \cdot X\bigr).
\end{equation*}

We now state the generalization of Proposition~\ref{PQ} to \tdm s.  It is a special case of more precise results of \cite{Le00}.  Its forward implication follows easily from the eigenvalues of the Casimir operator of $\da_m$ on the symbol modules, but the converse is more involved.  For convenience, define
\begin{equation} \label{DifflpPQlp}
   \Diff(\lp) := \Diff\bigl(F(\lambda), F(\lambda + p)\bigr), \quad
   \PQ_{\lp} := \PQ_{F(\lambda), F(\lambda+p)}.
\end{equation}

\begin{prop} \label{tdmPQ}
The projective quantization\/ $\PQ_\lp$ of the $\VRm$-module $\Diff(\lp)$ exists and is unique for all $\lambda$ \iff\ $p \not\in 1 + \frac{1}{m+1} \bN$.
\end{prop}

For $m=1$, $\PQ_\lp$ was computed in \cite{CMZ97}, and for $m>1$, $\PQ_{\lambda,0}$ was computed in \cite{LO99}.  The general formula may be found in \cite{DO01}.

The symbol modules of $\Diff(\lp)$ are independent of $\lambda$, so we write $\Symb^k(p)$ for the $k^\thup$ one, and $\Symb(p)$ for $\bigoplus_k \Symb^k(p)$.  Let $\sigma^{\lp}$ be the action of $\VRm$ on $\Diff(\lp)$, and let $\sigma_k^p$ be its action on $\Symb^k(p)$.  We have the following analogs of Definition~\ref{pi} and Lemma~\ref{piij}:

\begin{definition} \label{tdmpi}
Let $\pi^\lp$ be the action\/ $\PQ_\lp^{-1} \circ \sigma^\lp \circ \PQ_\lp$ of\/ $\VRm$ on\/ $\Symb(p)$.  Regard it as an infinite matrix with entries
$$ \pi^\lp_{ij}: \VRm \to \Hom_\bC\bigl(\Symb^j(p), \Symb^i(p)\bigr),\ \ i,j\in\bN. $$
\end{definition}

\begin{lemma} \label{tdmpiij}
\begin{enumerate}
\item[(a)]  The matrix $\pi^\lp$ is upper triangular.
\item[(b)]  Its diagonal entries are $\pi^\lp_{ii} = \sigma^p_i$.
\item[(c)]  For $i<j$, $\pi^\lp_{ij}$ is $\da_m$-covariant and zero on $\da_m$.
\end{enumerate}
\end{lemma}

To our knowledge, the five research problems have thus far been studied extensively only for the \tdm s.  We conclude this section by summarizing their status in this case.  This is in fact the general case for $m=1$, as there all \tfm s are direct sums of \tdm s.

\medbreak{\sc Problem~1.}
This consists in computing the $\pi^\lp_{ij}$ sufficiently explicitly for applications, for example to Problems~4 and~5.  For $m=1$, its solution follows from the solution of Problem~2 given in \cite{CMZ97}; the explicit formulas may be found in \cite{Co05}.  These formulas are valid also for \psdog s.  The resonant case was examined in \cite{Ga00, CS04}.  

For $m > 1$, those $\pi^\lp_{ij}$ with $p=0$ and $j-i=1$ or~$2$ were computed in \cite{LO99}, and each of the entries on the higher superdiagonals was shown to lie in a certain 2-dimensional space.  The $p\not=0$ cases are unexplored.

\medbreak{\sc Problem~2.}
Under composition of \dog s, $\bigoplus_\lp \Diff(\lp)$ is a filtered algebra whose commutative graded algebra is $\bigoplus_\lp \Symb(p)$.  The goal is to describe the associative algebra structure on $\bigoplus_\lp \Symb(p)$ obtained by pulling composition back via $\PQ_\lp$.  More precisely, composition is a map
\begin{equation} \label{comp}
   \Comp: \Diff(\lambda+p, q) \ot \Diff(\lp) \to \Diff(\lambda, p+q),
\end{equation}
and one wants to compute $\Comp^{\lambda,p,q} := \PQ_{\lambda,p+q}^{-1} \circ \Comp \circ (\PQ_{\lambda+p,q} \ot \PQ_\lp)$.  For $m=1$, this was carried out in \cite{CMZ97}.  We know of no results for $m>1$.

\medbreak{\sc Problem~3.}
For $m=1$ it turns out that the \tdm s all occur as symbol modules, so we wish to compute the $\VRm$-cohomology rings of the algebras $\bigoplus_\lp \Diff(\lp)$ and $\bigoplus_p \Symb(p)$.  The cohomology spaces of the symbols were computed in \cite{Go73}, and those of the \dog s were computed in \cite{FF80}.  The ring structure given by the cup product is essentially trivial on the symbols, but on the \dog s it is non-trivial.  The cup products on $H^1$ were computed in \cite{Co01, CS04}.  The higher cup products are not known.

For $m>1$, interesting results are obtained only by admitting more general \tfm s.  The sole result we know of in this direction is the computation of the 1-cohomology classes of $\Diff\bigl( \Symb^j(0), \Symb^i(0) \bigr)$ \cite{LO00}.

\medbreak{\sc Problem~4.}
The subjects of this problem are the \sq\ modules
\begin{equation} \label{PDOSQ}
   \SQ^k_l(\lp) := \Diff^k(\lp) \big/ \Diff^{k-l}(\lp)
\end{equation}
of $\VRm$.  Note that~$l$ is essentially the Jordan-H\"older length.  The basic question is to determine the equivalence classes of these modules.  This topic was introduced in \cite{DO97}, where the equivalence classes and $\VRm$-endomorphism rings of the modules $\Diff^2(\lambda,0)$ were determined.  For $m=1$, their work was extended to $\Diff^k(\lp)$ in \cite{Ga00}.  For $m>1$, it was extended to $\Diff^k(\lambda,0)$ in \cite{LMT96}.

Subquotients were first considered in \cite{LO99}, where the equivalence classes of the modules $\SQ^k_l(\lambda,0)$ are classified.  For $m=1$, the results extend to \psdog s.  In general, for high length~$l$, each module is equivalent only to its adjoint, while for low length, modules with the same composition series are usually equivalent.  In intermediate lengths, the equivalence classes are interesting.

\medbreak{\sc Problem~5.}
It is not certain that this problem can be completely solved, but progress has been made for $m=1$ and it would be interesting to try to replicate it for $m>1$.  Uniserial (\ie\ \cind) modules of length~2 are classified by the $\Ext^1$ groups between the elements of their composition series, and those of length~3 are classified by cup products in $\Ext^2$.  The classification of those of higher length amounts to solving the cup equation~(\ref{cup}).

For $m=1$, the length~2 and length~3 uniserial modules composed of \tdm s were computed in \cite{FF80} and \cite{Co01}, respectively.  Most of them, along with several of higher length, can be realized as \sq s of \psdog\ modules \cite{Co05}.  

For $m>1$, as in Problem~3 one gets interesting results only by admitting more general \tfm s as composition series elements.  To date only the $\Ext^1$ groups between the symbol modules $\Symb^k(0)$ computed in \cite{LO00} are known.

\section{Vec$\th\bR$} \label{R}

In this section we discuss Problems~1 through~5 in detail for $\VR$.  As mentioned, here it suffices to treat $F(\lambda)$ and $\Diff(\lp)$.  Note that $\Diff(\lp)$ is spanned by operators of the form $dx^p f(x) D^k$, where $f$ is a polynomial and $k \in \bN$.

By Proposition~\ref{tdmPQ}, $\PQ_\lp$ exists in the non-resonant cases $p\not\in 1+\bN/2$.  By~(\ref{VWsymbol}), the symbol module $\Symb^k(p)$ is equivalent to $F(p-k)$: the map $dx^p f(x) D^k \mapsto dx^{p-k} f(x)$ is an equivalence.  Thus $\PQ_\lp$ is an $\da_1$-equivalence
\begin{equation} \label{PQR}
   \PQ_\lp: \bigoplus_{k=0}^\infty F(p-k) \to \Diff(\lp).
\end{equation}

\subsection*{Problem~1.}
Here we want to compute the matrix entries $\pi^\lp_{ij}$ from Definition~\ref{tdmpi}.  By~(\ref{PQR}), they may be viewed as maps
$$ \pi^\lp_{ij}: \VR \to \Hom\bigl( F(p-j), F(p-i)\bigr). $$
By Lemma~\ref{tdmpiij}, they are $\da_1$-covariant and zero on $\da_1$.  Since $\VR/\da_1$ is $\da_1$-\irr\ with lowest weight~$2$ and \lwv\ $x^3D$, each entry is determined by its value on $x^3 D$, and this value must be a \lwv\ of weight~$2$.

It is simple to check that $\Hom\bigl( F(\mu), F(\mu+q)\bigr)$ has no \lwv s of weight~$2$ unless $q\in 2+\bN$, when up to a scalar it has one such, namely, $dx^q D^{q-2}$.  It follows that for $j-i \ge 2$, $\pi^\lp_{ij}$ is \dog-valued, maps $x^3 D$ to a multiple of $dx^{j-i} D^{j-i-2}$, and is completely determined by the multiple, while for $j-i = 1$ it is zero.  Therefore there are scalars $B^\lp_{ij}$ such that
\begin{equation} \label{Bij}
   \pi^\lp_{ij} (x^3 D) = 6 B^\lp_{ij} dx^{j-i} D^{j-i-2}.
\end{equation}

Let us describe the $\pi^\lp_{ij}$ in terms of transvectants and Bol operators, both classical objects.  First, for $\mu + \nu \not\in -\bN$ and $k \in \bN$, there exists a unique (up to a scalar) $\da_1$-covariant map
$$ J^{\mu, \nu}_k: F(\mu) \ot F(\nu) \to F(\mu + \nu + k), $$
the {\em transvectant.\/}  These maps were studied by Gordan in the nineteenth century and are closely related to Clebsch-Gordan coefficients.  They are essentially the same as the Rankin-Cohen brackets of modular forms.

Second, there exist non-scalar $\da_1$-covariant maps from $F(\mu)$ to $F(\nu)$ \iff\ $2\mu = 1-q$ and $2\nu = 1+q$ for some $q \in \bZ^+$.  In this case there is a unique (up to a scalar) such map, the {\em Bol operator\/}
\begin{equation} \label{Bol}
   \Bol_q := dx^q D^q: F({\ts \frac{1-q}{2}}) \to F({\ts \frac{1+q}{2}}).
\end{equation}
It is surjective, and its kernel is the $q$-dimensional $\da_1$-submodule of $F(\frac{1-q}{2})$.

Note that the adjoint action of $\VR$ on itself is naturally equivalent to $F(-1)$, via the identification $f(x) D \equiv dx^{-1} f(x)$.  Therefore $\Bol_3$ may be regarded as the unique $\da_1$-map from $\VR$ to $F(2)$, and as such it has kernel $\da_1$.  Thinking of the matrix entry $\pi^\lp_{ij}$ as a map from $\VR \ot F(p-j)$ to $F(p-i)$ and applying Lemma~\ref{tdmpiij}, we find that it must be a multiple of $J^{2, p-j}_{j-i-2} \circ (\Bol_3 \ot 1)$.

\subsubsection*{Computing the scalars}
So far we have seen that it is easy to compute the $\pi^\lp_{ij}$ up to the scalars $B^\lp_{ij}$.  The computation of these scalars is difficult.  It is a special case of Problem~2, but let us briefly describe the direct method used in \cite{Co01}.  

To begin with, for $q \in 2 + \bN$ define $\beta_q(\mu)$ to be the unique $\da_1$-map from $\VR$ to $\Diff(\mu, q)$ that is zero on $\da_1$ and carries $x^3 D$ to $6 dx^q D^{q-2}$.  Then~(\ref{Bij}) becomes
\begin{equation} \label{piijbeta}
   \pi^\lp_{ij} = B^\lp_{ij} \beta_{j-i}(p-j).
\end{equation}

Consider evaluating both $\pi^\lp_{ij}$ and $\beta_{j-i}(p-j)$ on the weight~$j-i$ vector field $x^{j-i+1} D$, and applying the resulting two weight~$j-i$ \dog s in $\Diff(p-j,j-i)$ to the \lwv\ $dx^{p-j}$ of $F(p-j)$.  This will give two multiples of the \lwv\ $dx^{p-i}$ of $F(p-i)$, and the ratio of these multiples is $B^\lp_{ij}$.  Let us temporarily call these multiples $C_\pi$ and $C_\beta$, respectively.

It is easy to compute $C_\beta$.  To compute $C_\pi$, we must go back to the definition of $\pi^\lp$ as $\PQ^{-1}_\lp \circ \sigma^\lp \circ \PQ_\lp$.  Since $\PQ_\lp$ preserves \lwv s and symbols, it maps $dx^{p-j}$ to $dx^p D^j$.  Using this and the fact that $\PQ_\lp$ is an $\da_1$-map, one finds that the image of $dx^p D^j$ under $\sigma^\lp(x^{j-i+1} D)$ is $C_\pi dx^p D^i$, {\em modulo the image of $\sigma^\lp(x^2 D)$,\/} the action of the raising operator in $\da_1$.  Continuing in this vein, one obtains $C_\pi$ with the help of the element $P(\da_1)(x^{j-i+1} D)$ of the {\em step algebra\/} $S(\VR, \da_1)$.  (Here $P(\da_1)$ is the {\em extremal projector\/} of $\da_1$.)

The explicit formula for $B^\lp_{ij}$ is long and we will not include it here.  In the notation of~(3) and~(4) of \cite{Co05}, it is $b_{p-i,p-j}(\lp)$.

\subsubsection*{The resonant case}
For $p\in 1+\bN/2$, $\Diff(\lp)$ is in general not completely reducible under $\da_1$, and there is no projective quantization of the form~(\ref{PQR}).  However, there is a {\em resonant projective quantization\/} $\o\PQ_\lp$, not far removed in spirit from the usual projective quantization \cite{Ga00, CS04}.  

We have mentioned that the resonant cases play a role in the study of the non-resonant cases.  By this we mean that the non-resonant matrix entries $\pi^\lp_{ij}$ take values in the resonant modules $\Diff(p-j, j-i)$.

To construct $\o\PQ_\lp$, one first checks that the Casimir operator of $\da_1$ acts on $F(\mu)$ by the scalar $\mu^2-\mu$, so it has the same eigenvalues on $F(\mu)$ and $F(1-\mu)$.  Therefore, in the resonant case some of the generalized Casimir eigenspaces contain two symbol modules: $\Symb^k(p) \cong F(p-k)$ and $\Symb^{2p-k-1}(p) \cong F(1-p+k)$ are in the same generalized eigenspace for $k < p-1/2$.  In general, $\da_1$ does not act semisimply on this generalized eigenspace: it is the injective $\da_1$-module of $F(p-k)$.

This injective module does split under the affine subalgebra $\db_1 = \du_1 \oplus \dl_1$: there is a 1-parameter family of $\db_1$-injections from $F(p-k) \oplus F(1-p+k)$ to $\Diff(\lp)$ which are the identity on symbols.  Thus the generalized eigenspaces of the Casimir operator lead to the following result: there is a $\lf p-1 \rf$-parameter family of {\em affine quantizations\/} ($\db_1$-equivalences which are the identity on symbols)
$$ \o\PQ_\lp: \bigoplus_{k=0}^\infty F(p-k) \to \Diff(\lp) $$
which induce projective quantizations of both the quotient $\Diff(\lp) / \Diff^{\lf p-1 \rf}(\lp)$ and the submodule $\Diff^{\lf p-1/2 \rf}(\lp)$.  (Here $\lf\cdot\rf$ denotes the integer part.)

The aim of \cite{CS04} was to find the most natural choice of $\o\PQ_\lp$ in this family and compute the associated matrix entries $\o\pi^\lp_{ij}$.  Some properties of the matrix entries are true for all choices of $\o\PQ_\lp$ in the family; let us begin by listing them.

The fact that $\o\PQ_\lp$ induces the projective quantization on the above sub- and quotient \dog\ modules implies that those $\o\pi^\lp_{ij}$ in the {\em non-resonant triangles,\/} where either $i<j<p$ or $p-1<i<j$, are still given by~(\ref{piijbeta}).

The fact that $\o\PQ_\lp$ is a $\db_1$-equivalence implies that all non-diagonal $\o\pi^\lp_{ij}$ are $\db_1$-covariant maps which are zero on $\db_1$.  Furthermore, the fact that $\o\PQ_\lp$ respects the generalized Casimir eigenspaces means that only those $\o\pi^\lp_{ij}$ on the {\em antidiagonal,\/} where $i+j = 2p-1$, can be non-zero on $\da_1$.

The entries above the diagonal are partitioned by three regions: the two non-resonant triangles and the {\em resonant rectangle,\/} where $j\ge p$ and $i\le p-1$.  The antidiagonal is contained in the resonant rectangle.  Entries in this rectangle but off the antidiagonal must be zero on $\da_1$, but in general they are not $\da_1$-covariant.  However, using the cup equation together with the fact that the entries in the non-resonant triangles on the first superdiagonal are zero, it can be shown that the entries in the resonant rectangle adjacent to the antidiagonal, those with $i+j$ equal to $2p-2$ or $2p$, are given by~(\ref{piijbeta}).

Now we turn to the question of the most canonical choice of $\o\PQ_\lp$.  The exceptional resonant cases which do admit quantizations have to do with {\em conjugation,\/} a $\VR$-equivalence between adjoint modules:
\begin{equation} \label{conj}
   \C_{\mu,q}: \Diff(\mu,q) \to \Diff(1-\mu-q, q).
\end{equation}
For $2\mu+q=1$, $\C_{\mu,q}$ is an involution defining a $\VR$-splitting of $\Diff(\mu,q)$ along even and odd order.  This is the {\em self-adjoint\/} case.

In those self-adjoint resonant cases with $q\in \bZ^+$ (where $\Bol_q(\mu)$ exists), the doubled Casimir eigenvalues are split by the eigenspaces of the involution $\C_{\mu,q}$, and so $\Diff(\mu,q)$ has a unique projective quantization compatible with $\C_{\mu,q}$.  The resonant antidiagonal matrix entries take values precisely in these cases.  In \cite{CS04} we sought $\o\PQ_\lp$ such that these entries take values in one of the eigenspaces of conjugation.  This condition can be met: it determines $\o\PQ_\lp$ uniquely for $p\in \frac{3}{2} + \bN$, and up to a single parameter for $p\in 1 + \bN$.

Let us conclude this digression on the resonant case with a brief description of the form of the antidiagonal entries, as they are quite lovely.  For $\mu$ arbitrary and $q\in\bN$, the {\em affine Bol operator\/} $dx^q D^q$ in $\Diff^q(\mu, q)$ is $\db_1$-invariant and has $\VR$-invariant symbol.  Therefore its {\em coboundary\/} $\partial(dx^q D^q)$ is a $\db_1$-relative $\Diff^{q-1}$-valued 1-cocycle.  In the self-adjoint case, $dx^q D^q$ is $\Bol_q$, and its coboundary is a multiple of the $\Diff^{q-2}$-valued 1-cochain $\beta_q(\mu)$.  In general, $\partial(dx^q D^q)$ is a linear combination of $\beta_q(\mu)$ and a certain $\Diff^{q-1}$-valued 1-cochain $\alpha_q(\mu)$, uniquely determined by its symbol and the condition that it have no {\em subsymbol.\/}  (Since the non-resonant entries $\o\pi^\lp_{ij}$ on the first superdiagonal are zero, elements of $\Diff(\mu,q)$ have $\VR$-invariant subsymbols as well as symbols, except in order~$q$: see Problem~4 below.)

This 1-cochain $\alpha_q(\mu)$ has an analytic continuation to the self-adjoint value of $\mu$, and the antidiagonal entry $\o\pi^\lp_{ij}$ is a multiple of $\alpha_{j-i}(p-j)$.  The multiple, as well as the other entries in the resonant rectangle, may be computed by taking the resonant limit of the non-resonant case.

\subsection*{Problem~2.}
The goal here is to compute the map $\Comp^{\lambda,p,q}$ defined below~(\ref{comp}).  By~(\ref{PQR}), it has range and domain
\begin{displaymath}
   \Comp^{\lambda,p,q}: \Bigl(\bigoplus_{i=0}^\infty F(q-i) \Bigr)
   \ot \Bigl(\bigoplus_{j=0}^\infty F(p-j) \Bigr)
   \to \bigoplus_{k=0}^\infty F(p+q-k).
\end{displaymath}

By definition, $\Comp^{\lambda,p,q}$ is $\da_1$-covariant, and it exists only when no resonant \dog\ modules are involved.  Therefore its constituent maps
\begin{displaymath}
   \Comp^{\lambda,p,q}_{i,j,k}: F(q-i) \ot F(p-j) \to F(p+q-k)
\end{displaymath}
are $\da_1$-maps, hence multiples of transvectants: for some scalars $t^{\lambda,p,q}_{i,j,k}$,
\begin{displaymath}
   \Comp^{\lambda,p,q}_{i,j,k} = t^{\lambda,p,q}_{i,j,k}\th J^{q-i,p-j}_{i+j-k}.
\end{displaymath}

The main result of \cite{CMZ97} is the computation of the $t^{\lambda,p,q}_{i,j,k}$ and their symmetries: a tour de force.  In addition to several order~2 symmetries related to conjugation and the {\em Adler trace\/} (or {\em noncommutative residue\/}), they possess an order~3 symmetry arising from a triality of the transvectants and the Adler trace.

This project can be carried out in the resonant case with $\o\PQ$ in place of $\PQ$: see \cite{CS04}, where expressions for the resulting scalars are given as limits of the non-resonant scalars.  In this setting the transvectants are replaced by $\dsl_2$-covariant maps from the tensor product of two injective modules to a third.

\subsection*{Problem 3.}
Here we describe the results of \cite{Go73} and \cite{FF80}, which are extremely beautiful.  We begin with the $\VR$-cohomology of the \tdm s.  Clearly $H^0 F(\lambda)$ is zero unless $\lambda=0$, when it is $\bC\cdot 1$.  It is instructive to describe $H^1 F(\lambda)$ explicitly: it is zero unless $\lambda$ is~$0$, $1$, or $2$, where, regarding $\VR$ as $F(-1)$ as below~(\ref{Bol}),
\begin{equation} \label{VR1coho}
   H^1 F(0) = \bC dx D, \quad
   H^1 F(1) = \bC dx^2 D^2, \quad
   H^1 F(2) = \bC dx^3 D^3.
\end{equation}
All of these 1-cocycles are $\du_1$-relative and $\db_1$-covariant.  Moreover, $dx^2 D^2$ is $\db_1$-relative, and $dx^3 D^3$, being $\Bol_3$, is $\da_1$-relative.

Let us give some idea of how~(\ref{VR1coho}) is proven.  The following lemma is useful.

\begin{lemma} \label{FCL}
Any $r$-cohomology class of\/ $\VR$ taking values in\/ $F(\lambda)$, $\Diff(\lp)$, or\/ $\Hom\bigl(F(\lambda), F(\lambda+p)\bigr)$ is $\du_1$-relative and $\db_1$-covariant, and its contraction by the Euler operator $E=xD$ is a $\db_1$-relative $(r-1)$-class.
\end{lemma}

To prove~(\ref{VR1coho}), let $\omega$ be a non-zero $\du_1$-relative $\db_1$-covariant $F(\lambda)$-valued 1-cocycle of $\VR$.  Then $\omega$ is determined by $\omega(x^{\lambda+1}D)$, a multiple of $dx^\lambda$.  In particular, $\lambda\in\bN$.  If $\lambda > 1$, then $\omega(\da_1)=0$ together with $\partial\omega=0$ imply that $\omega$ is $\da_1$-relative, giving $\lambda=2$.  The rest is easy.

\subsubsection*{Results of\/ \cite{Go73}}
Here $H^r F(\lambda)$ is computed.  Define $\ep^\pm(r) := \frac{1}{2}(3r^2\pm r)$, the Euler polynomials.  First we give the $\db_1$-relative cohomology $H_{\db_1}$: for $r\in\bN$, there exist $\db_1$-relative $F\bigl(\ep^\pm(r)\bigr)$-valued $r$-cocycles $\phi^\pm_r$ of $\VR$ such that
\begin{equation} \label{bGo}
   H_{\db_1}^r F(\lambda) = \bC \phi^\pm_r
   \mbox{\rm\ \ if $\lambda = \ep^\pm(r)$, and zero otherwise.}
\end{equation}

For example, $\phi^-_0 = \phi^+_0 = 1$, $\phi^-_1 = dx^2 D^2$, and $\phi^+_1 = \Bol_3$.  In fact, all the $\phi^\pm_r$ except for $\phi^-_1$ are $\da_1$-relative.  The full cohomology consists of the $\db_1$-relative cohomology together with its cup product with the non-$\db_1$-relative 1-cocycle $dx D$:

\begin{theorem} \cite{Go73}
$\phi^\pm_r$ and $(dx D)\cup \phi^\pm_{r-1}$ are a basis for $\bigoplus_{\lambda\in\bC} H^r F(\lambda)$.
\end{theorem}

Note that $\bigoplus_\lambda F(\lambda)$, the sum of all the symbol modules, is a commutative algebra.  It follows from~(\ref{bGo}) that the associated cup product is trivial on $\db_1$-relative cohomology.  For example, $\phi^-_1 \cup \phi^+_1 = -2\partial(dx^4 D^4)$.

\subsubsection*{Results of\/ \cite{FF80}}
Here $H^r \Hom\bigl(F(\lambda), F(\lambda+p)\bigr)$ is computed.  Adapting Lemma~7.4 of \cite{Co08} shows that the inclusion of $\Diff(\lp)$ induces an isomorphism in cohomology, so we will only discuss $H^r \Diff(\lp)$.

It is again instructive to begin with $H^0$ and $H^1$.  One checks that $H^0 \Diff(\lp)$ is zero unless either $p=0$, when it is $\bC\cdot 1$, or $(\lp) = (0,1)$, when it is $\bC \Bol_1$.  For $p\in\bN$, the lift of the 0-cocycle $\phi^\pm_0=1$ in $F(0)$ via $\o\PQ_\lp$ is $dx^p D^p$.  Thus the 0-cocycles of the $\Diff(\lp)$ are those lifts of the 0-cocycles of the $F(\lambda)$ which are still cocycles after lifting.  This is the rough picture in all degrees.

Now consider $H^1_{\da_1}$, the $\da_1$-relative case.  There are no $\Diff(\lp)$-valued $\da_1$-relative 1-cochains unless $p\in 2+\bN$, when up to a scalar there is exactly one: the map $\beta_p(\lambda)$ defined above~(\ref{piijbeta}).  Note that it is the $\o\PQ_\lp$-lift of $\phi^+_1$ from $F(2)$.

The symbol of the coboundary $\partial\beta_p$ is an $\da_1$-relative 2-cocycle.  Using \cite{Go73}, one finds that it is proportional to one of $\phi^\pm_2$, which are $F(5)$- and $F(7)$-valued.  Therefore $\partial\beta_p$ is zero or of order~$p-5$ or $p-7$.  In particular, if $p<5$ it is zero.

For $p\ge 5$, $\partial\beta_p$ is the $\o\PQ_\lp$-lift of a linear combination of $\phi^\pm_2$.  The coefficient of $\phi^-_2$ is essentially $(2\lambda+p-1)^2 - (3p+1)$; whenever this is zero, $\partial\beta_p$ is of order~$p-7$.  For $p\ge 7$, the coefficients of $\phi^\pm_2$ are never simultaneously zero.  This proves that $\beta_p$ is a cocycle only in the following cases: for all $\lambda$ if $p=2$, $3$, or~$4$; for $\lambda = -4$ or~$0$ if $p=5$; and for $2\lambda = -5 \pm \sqrt{19}$ if $p=6$.

Now $\beta_p$ is not always non-trivial: in the self-adjoint case it is a multiple of $\partial\Bol_p$.  However, in the $p=1$, $2$, $3$, and~$4$ self-adjoint cases, an appropriate lift of the other $\db_1$-relative tensor density 1-cocycle, the $F(1)$-valued $\phi^-_1$, gives a non-trivial $\Diff(\lp)$-valued cocycle of order~$p-1$.  This lift is not exactly the $\o\PQ_\lp$-lift (unless $p=1$ or~$2$); rather, it is the map $\alpha_p$ occurring in the resonant case.

The result is that $H^1_{\db_1} \Diff(\lp)$ is $\bC$ if $p=1$ and $\lambda=0$; $p=2$, $3$, or~$4$; $p=5$ and $\lambda = -4$ or~$0$; or $p=6$ and $2\lambda = -5 \pm \sqrt{19}$.  Otherwise it is zero.

We now turn to $H_{\db_1}^r \Diff(\lp)$.  For $p \ge \ep^\pm(r)$, it is clear from \cite{Go73} that we can construct $\db_1$-relative $\Diff(\lp)$-valued $r$-cochains $\t\phi^\pm_r(\lp)$ of order $p-\ep^\pm_r$ whose symbols are $\phi^\pm_r$ and whose coboundaries are of minimal order.  (These cochains are $\da_1$-relative for $r\ge 2$, but in general $\da_1$-relativity alone does not determine them.)  If $\t\phi^\pm_r$ is not a cocycle, then necessarily the symbol of its coboundary is non-trivial, so said symbol must be cohomologous to one of $\phi^\pm_{r+1}$.  Therefore we can construct the $\t\phi^\pm_r$ so that for some coefficients $M^\pm_r(\lp)$ and $P^\pm_r(\lp)$, we have
\begin{equation} \label{FF}
   \partial \t\phi^\pm_r = M^\pm_r \t\phi^-_{r+1} + P^\pm_r \t\phi^+_{r+1}.
\end{equation}
For example, $\t\phi^\pm_0 = dx^p D^p$, $\t\phi^-_1 = \alpha_p$, and $\t\phi^+_1 = \beta_p$.  It is understood that $\t\phi^\pm_{r+1}$ is to be replaced with zero for $p < \ep^\pm(r+1)$.

For generic $\lambda$, the coefficients $M^\pm_r$ and $P^\pm_r$ are non-zero.  In this case we have the following picture: for $p< \ep^-(r)$, no $\Diff(\lp)$-valued $r$-cocycle has non-trivial symbol.  For $\ep^-(r) \le p < \ep^+(r)$, $\t\phi^-_r$ is a coboundary.  For $p \ge \ep^+(r)$, $\t\phi^-_r$ and $\t\phi^+_r$ are proportional in cohomology.  For $\ep^+(r) \le p < \ep^-(r+1)$, $\t\phi^+_r$ is a non-trivial cocycle, but for $p \ge \ep^-(r+1)$ it is not a cocycle.

In summary, for generic $\lambda$, $H_{\db_1}^r \Diff(\lp)$ is $\bC \t\phi^+_r$ for $\ep^+(r) \le p <\ep^-(r+1)$, and zero otherwise.  However, the roots of $M^\pm_r$ and $P^\pm_r$ give special cases.  For example, for $\ep^-(r+1) \le p < \ep^+(r+1)$ and $M^+_r(\lp)=0$, $\t\phi^+_r$ is a cocycle, in general non-trivial.  For $H^1_{\db_1}$, this gives the special cases at $p=5$ and~$6$.

Similarly, for $\ep^-(r) \le p < \ep^-(r+1)$ and $M^+_{r-1}(\lp)= 0$, $\t\phi^+_r$ is a coboundary but $\t\phi^-_r$ is in general a non-trivial cocycle.  For $H^1_{\db_1}$, these are the self-adjoint cases.  See Theorem~4.2B and the ``parabola picture'' (Figure~3) of \cite{FF80} for $H^r_{\db_1}$.

Just as for \tdm s, $H^r$ is obtained from $H^r_{\db_1}$ by cupping with the unique non-$\db_1$-relative 1-cocycle.  To be precise, let $\theta(\lambda)$ be the lift of $dx D$ to the $\Diff^0(\lambda,0)$-valued 1-cocycle $f D \mapsto f'$.  Then a basis of $H \Diff(\lp)$ is given by a basis of $H_{\db_1} \Diff(\lp)$, together with its cup product with $\theta(\lambda)$.

\subsubsection*{Remarks}
It would seem worthwhile to find a proof of the results of \cite{Go73} and \cite{FF80} featuring $\db_1$- and $\da_1$-relativity and the spectral sequence associated to the order filtration of the \dog\ modules.  

In \cite{FF80}, the modules with cohomology have two continuous parameters, while the modules $\Diff(\lp)$ have only one, namely, $\lambda$, as $p$ must be in $\bN$ to give cohomology.  We obtain a second parameter by passing to quotients of {\em \psdog\/} (\psidog) modules, as follows.   

For $k\in\bC$, the module of \psidog s of order~$\le k$ from $F(\lambda)$ to $F(\lambda+p)$ is
\begin{displaymath}
   \Psi^k(\lp) := \Bigl\{ \sum_{n=0}^\infty dx^p f_n(x) D^{k-n}: f_n \in \bC[x] \Bigr\}.
\end{displaymath}
Write $\Psi^{k+\bN}$ for $\bigoplus_{j=0}^\infty \Psi^{k+j}$.  Note that $\Diff(\lp)$ is the quotient $\Psi^\bN(\lp)/\Psi^{-1}(\lp)$, and consider the quotients $\Psi^{p+\bN}(\lp)/ \Psi^{p-n-1}(\lp)$ with $\lp\in\bC$ and $n\in\bN$.  They have composition series $F(n), F(n-1), \ldots$, and it seems likely that their cohomology mirrors that of the 2-parameter family studied in \cite{FF80}.  For example, their $H^1_{\db_1}$ is $\bC$ if $n=1$ and $(\lp)$ are self-adjoint; if $n=2$, $3$, or~$4$ for all $(\lp)$; and if $n=5$ or~$6$ and $(2\lambda+p-1)^2 - (3p+1)$ is zero.  Otherwise it is zero.

\subsubsection*{Cup products}
The cup product associated to the algebra $\bigoplus_\lp \Diff(\lp)$ is not trivial on $H_{\db_1}$.  It is known only on $H^1_{\db_1}$ \cite{Co01}, where in the $\da_1$-relative case, the key observation is that $\beta_q(\lambda+p) \cup \beta_p(\lambda)$ is trivial \iff\ it is a multiple of $\partial\beta_{p+q}(\lambda)$.  Both of these 2-cocycles are linear combinations of $\t\phi^\pm_2(\lambda, p+q)$, and so one must compute their coefficients.  For $p+q<7$, $\t\phi^+_2$ is zero so the cup product is usually trivial.  For $p+q \ge 7$, triviality imposes a condition on~$\lambda$; for example, at $p=q=4$ the cup product is trivial only at $2\lambda = -7 \pm \sqrt{39}$.

In \cite{Co01} we treated only \dog\ modules.  The correct level of generality seems to be to treat the quotients of \psidog\ modules defined above.

\subsection*{Problem 4}
Recall $\SQ^k_l(\lp)$ from~(\ref{PDOSQ}), and extend the definition to \psidog\ modules in the obvious way, so that $k$ becomes a continuous parameter.  Our goal is to describe the equivalence classes and $\VR$-endomorphism rings of these modules.  We discuss only the non-resonant case; see \cite{Ga00} for the resonant case.

$\SQ^k_l(\lp)$ has composition series $\{F(p-k+j): 0\le j < l\}$, so $l$ and $p-k$ are invariants.  Applying $\PQ_\lp$, we may regard the $\VR$-action as an upper triangular $l \times l$ matrix with entries given by~(\ref{piijbeta}).  The $\da_1$-endomorphism ring consists of the diagonal matrices with scalar entries, and the $\VR$-endomorphism ring is isomorphic to $\bC^e$, where~$e$ is the number of \ind\ summands of the module.  We will denote the latter by $\End^k_l(\lp)$.

These observations lead to the following results, which are new for $p\not=0$ except in the case of genuine \dog\ modules (not quotients).  We will only outline them here; the details will make up part of a future paper treating $\bR$, $\Roo$, and $\bR^m$ together.  Henceforth fix~$l$ and $p-k$, leaving only $\lambda$ and $p$ free.

\subsubsection*{Length 2}
$\SQ^k_2$ splits under $\VR$ as $\bigoplus_{j=0}^1 F(p-k+j)$ in the non-resonant case, as the matrix entries $\pi^\lp_{j-1,j}$ are always zero.  Hence these modules are all equivalent and have $\End^k_2 = \bC^2$.  In particular, \psidog s of order~$k$ have $\VR$-invariant $F(p-k+1)$-valued {\em subsymbols\/} for $p\not= k$.

\subsubsection*{Length 3}
$\SQ^k_3$ splits under $\VR$ as $\bigoplus_{j=0}^2 F(p-k+j)$ \iff\ $\pi^\lp_{k-2,k}=0$, which occurs \iff\ the scalar $B^\lp_{k-2,k}$ is zero (see~(\ref{piijbeta})).  Otherwise it is the direct sum of $F(p-k+1)$ and the \ind\ module composed of $F(p-k)$ and $F(p-k+2)$.  Thus there are two equivalence classes, determined by whether $B^\lp_{k-2,k}$ is zero or not.  They have $\End^k_3 = \bC^3$ and $\bC^2$, respectively.

In order to describe the solutions of $B^\lp_{k-2,k} =0$, we introduce the variable 
\begin{equation} \label{C}
  C := (\lambda + p/2 -1/2)^2.
\end{equation}
Conjugation~(\ref{conj}) preserves $C$, so the equivalence classes can only depend on $(C,p)$.  By~(6) of \cite{Co05}, $B^\lp_{k-2,k}=0$ is a line in $(C,p)$-space.

\subsubsection*{Length 4}
The equivalence classes of the modules $\SQ^k_4$ are determined by which of the curves $B^\lp_{k-2,k}=0$, $B^\lp_{k-3,k-1}=0$, and $B^\lp_{k-3,k}=0$ the point $(C,p)$ is on.  By~(6) and~(7) of \cite{Co05}, the first two curves are lines and the third is the union of two lines.  As long as none of these lines intersect in resonant points, there are at least seven equivalence classes, and eight if the three curves have a common point.  For most $k$ they do not, but in two cases they do: $\Diff^3(-\frac{2}{3},\frac{7}{3})$ and its dual (via the Adler trace) $\SQ^{-1}_4(\frac{5}{3}, -\frac{7}{3})$.  These cases are linked to the {\em Grozman operator.\/}

\subsubsection*{Length 5}
As before, two modules $\SQ^k_5$ can be equivalent only if they are either both on or both off each of the curves $B^\lp_{k-i,k-j}=0$ for $(i,j)$ equal to $(2,0)$, $(3,1)$, $(4,2)$, $(3,0)$, $(4,1)$, and $(4,0)$.  The first five of these were described above.  By~(8) of \cite{Co05}, $B^\lp_{k-4,k}$ is a parabola in $(C,p)$-space.

In this case there is a hitherto unnoticed phenomenon, visible only when all three of $\lambda$, $p$, and $k$ are allowed to vary continuously: the above ``same vanishing'' conditions are not sufficient for equivalence.  There are two continuous invariants:
\begin{equation} \label{invts}
   B^\lp_{k-4,k-2} B^\lp_{k-2,k} \big/ B^\lp_{k-4,k}, \qquad
   B^\lp_{k-4,k-1} B^\lp_{k-3,k} \big/ B^\lp_{k-4,k} B^\lp_{k-3,k-1}.
\end{equation}

The level curves in $(C,p)$-space of the first of these invariants form a pencil of conics through four fixed points depending only on $k$, and the level curves of the second form a 1-parameter family of cubics.  Since cubics intersect conics in six points, generic equivalence classes in length~5 consist of six pairs of adjoint modules.  

Some interesting points remain unclarified, for example the nature of the family of cubics, and the significance of the four points defining the pencil of conics.  Also, the coordinate system in which the pencil of conics is ``nicest'', the one in which the four points are inscribed in a circle, is intriguing.  In it the cubics are reduced, and its axes seem to have some meaning.

\subsubsection*{Length $\ge$ 6}
Here we expect that at least generically, each module is equivalent only to its adjoint.  However, in length~6 there may be special values of $k$ at which there are other equivalences.  This is because any two length~6 modules whose two pairs of length~5 \sq s are equivalent are themselves equivalent (a consequence of the fact that $\pi^\lp_{k-5,k}$ is not a cocycle).

\subsubsection*{Lacunary \sq s}
Let us introduce a variation of Problem~4 which has not yet been considered.  In the \sq s above, the symbol modules in the composition series are always ``consecutive'', but there are other $\VR$-\sq s.  We discuss only the simplest case: define
\begin{displaymath}
   \SQ^k_{\lac}(\lp) := \PQ_\lp\bigl(F(p-k) \oplus F(p-k+2) \oplus F(p-k+4)\bigr),
\end{displaymath}
the projective quantization of the $k^\thup$, $(k-2)^\ndup$, and $(k-4)^\thup$ symbol modules.  This is of course an $\da_1$-\sq\ of $\Psi^k(\lp)$, but in light of the invariant subsymbol resulting from the length~2 case above, it is in fact a $\VR$-\sq.  

Fixing~$p-k$, the equivalence classes of the $\SQ^k_{\lac}(\lp)$ are given by the ``same vanishing'' condition on $B^\lp_{k-2,k}$, $B^\lp_{k-4,k-2}$, and $B^\lp_{k-4,k}$, along with the first invariant of~(\ref{invts}).  Thus in the region where none of these three scalars vanish, the equivalence classes form the pencil of conics in $(C,p)$-space discussed above.

\subsection*{Problem 5}
Define a $(\lambda; p_1,\ldots,p_n)$ {\em extension\/} to be a $\VR$-module $W$ with an invariant flag $W=W_0\supset W_1\supset \cdots \supset W_{n+1} =0$ such that $W_i/W_{i+1}$ is equivalent to $F(\lambda+ p_1 + \cdots + p_i)$ for all~$i$.  Such an extension is called {\em uniserial\/} if all of its \sq s are indecomposable, \ie\ if $W_i/W_{i+2}$ is \ind\ for all~$i$.

The uniserial $(\lambda; p)$ extensions are classified by $\bP H^1\Diff(\lp)$ (see Problem~3): we get one for all $\lambda$ at $p=0$, $2$, $3$, or~$4$, two at $(0; 1)$, one at the dual pairs $(-2 \pm 2; 5)$ and $(\frac{1}{2}(-5 \pm \sqrt{19}); 6)$, and no others (see Table~1 of \cite{FF80}).

Comparing Problems~1 and~3, we see that the 1-cocycles defining the $\da_1$-split $(\lambda; p)$ extensions, those with $p=2$, $3$, $4$, or~$6$ and $C \not= 0$ (see~(\ref{C})), occur as the $p^\thup$ superdiagonal matrix entries $\pi^{\mu, q}_{k-p,k}$ (the $p=5$ cases are blocked by resonance).

Uniserial $(\lambda; p, q)$ extensions exist when there are uniserial $(\lambda; p)$ and $(\lambda+p; q)$ extensions and the cup product of the associated cocycles is trivial, in which case they are classified by $H^1\Diff(\lambda, p+q)$.  In the $\da_1$-split case, the results are roughly as follows (see \cite{Co01} and the subsection on cup products above).

There is a 1-parameter family of $\da_1$-split $(\lambda; 2,2)$ extensions for most $\lambda$, as $\beta_4(\lambda)$ is a cocycle.  For $p+q = 5$ or~$6$, there exists a unique $(\lambda; p,q)$ extension for most $\lambda$.  For $p+q=7$ or~$8$, one exists only for special $\lambda$: at $(\frac{1}{2}; 4,3)$, its dual $(-\frac{13}{2}; 3,4)$, and the dual pair $(\frac{1}{2}(-7 \pm \sqrt{39}); 4,4)$.

Uniserial modules of length~$\ge 4$ are difficult to construct by hand from~(\ref{cup}), but one can look for them ``in nature''.  For example, it is observed in \cite{FF80} that the $(\lambda; 2)$ extensions are subquotients of the modules $\Lambda^2 F(\mu)$.  Olivier Mathieu posed the following question: which uniserial modules arise as \sq s of the modules $\Psi^k(\lp)$?  Using Problem~1, one finds essentially all known $\db_1$-split cases in this context, along with several more which are new \cite{Co05}.  Taking $k=p-\lambda$ and $C=0$ gives a 1-parameter family of $(\lambda; 2,2,\ldots)$ extensions of infinite length for each $\lambda$, generically uniserial.  (Heuristic parameter-counting arguments fail to predict this family; a conceptual explanation of its existence would shed new light on Problem~3.)  Taking also $B^\lp_{k-2,k} = B^\lp_{k-2l-2,k-2l} = 0$ gives uniserial $(\lambda; 4,2,2,\ldots,2,4)$ extensions of length~$l$ at the special values $(2\lambda +2l+1)^2 = 4l^2+3$.  (The $l=2$ and~$3$ cases are the $(\lambda; 6)$ and $(\lambda; 4,4)$ extensions.)  There are also uniserial $(\frac{1}{8}(-27 \pm \sqrt{649}); 3,3,2)$ and $(\frac{1}{16}(-67 \pm \sqrt{3529}): 3,2,2,2)$ \sq s.

Let us mention that in the case of the Virasoro Lie algebra, there is a uniserial module composed of two trivial modules and the pinned \tdm\ in which the central element does not act by zero \cite{MP92}.  A natural realization would be interesting.

\section{Vec$\th\bR^m$}

Now fix $m > 1$.  Here much less is known (in particular, we will say nothing about Problem~2).  We shall discuss essentially only the \tdm s, and we shall restrict to the non-resonant case $p\not\in 1 + \frac{1}{m+1}\bN$ (see Proposition~\ref{tdmPQ}).  However, first let us say that it would be interesting to know if all injective modules of \tfm s arise as $\da_m$-submodules of the modules $\Diff\bigl(F(V), F(W)\bigr)$ in the resonant case.  We shall also restrict to \dog s, although eventually \psidog s should be considered.

By~(\ref{VWsymbol}), we have $\Symb^k(p) \cong F\bigl( L(-ke_1) \ot \bC_p \bigr)$ as $\VRm$-modules.  Therefore (see~(\ref{genPQ}, \ref{DifflpPQlp})) the projective quantization is an $\da_m$-equivalence
\begin{equation*}
   \PQ_\lp: \bigoplus_{k=0}^\infty F\bigl( L(-ke_1) \ot \bC_p \bigr) \to \Diff(\lp).
\end{equation*}

\subsection*{Problem 1}
As over $\bR$, it is easy to prove that the matrix entries $\pi^\lp_{ij}$ of Definition~\ref{tdmpi} are \dog-valued.  By Lemma~\ref{tdmpiij}, they are $\da_m$-maps
\begin{equation*}
   \pi^\lp_{ij}: \VRm / \da_m \to \Diff\bigl( F\bigl( L(-je_1) \ot \bC_p \bigr), \th
   F\bigl( L(-ie_1) \ot \bC_p \bigr) \bigr).
\end{equation*}

In the non-resonant case, the \dog\ module on the right splits under $\da_m$ as the sum of its symbol modules.  By~(\ref{VWsymbol}), these symbol modules are
\begin{equation*}
   \Symb^k \bigl( F\bigl( L(-je_1) \ot \bC_p \bigr), \th
   F\bigl( L(-ie_1) \ot \bC_p \bigr) \bigr) \cong
   F\bigl( L(-ke_1) \ot L(je_m) \ot L(-ie_1) \bigr).
\end{equation*}
Note that the Euler weight of the $\du_m$-kernel of this module is $j-i-k$.

Recall that Verma modules have unique \irr\ quotients, so the modules $F\bigl(L(\lambda)\bigr)$, being dual to Verma modules, have unique \irr\ submodules.  As an $\da_m$-module, $\VRm/\da_m$ has \lwv\ $x_m^2 D_1$ and is equivalent to the unique \irr\ $\da_m$-submodule of $F\bigl(L(2e_m-e_1)\bigr)$ (which is in fact a proper submodule).  It follows that $\pi^\lp_{ij}(x_m^2 D_1)$ determines $\pi^\lp_{ij}$ and is an $\da_m$-\lwv\ of weight $2e_m-e_1$.  Since it has Euler weight~$1$, $\pi^\lp_{ij}$ must take values in the image of the $(j-i-1)^\stup$ symbol module under the projective quantization.

The dimension of the space of \lwv s of weight $2e_m-e_1$ in this symbol module is the $\dl_m$-multiplicity of $L(2e_m-e_1)$ in $L(-(j-i-1)e_1) \ot L(je_m) \ot L(-ie_1)$.  By the PRV lemma and properties of minuscule weights, this multiplicity is two in general, one if $i=0$ or $j-i=1$, and zero if both hold.

Using the well-known normal ordering notation in which $\xi_i$ represents the symbol of $D_i$ (see, \eg\ \cite{LO99}), one can write the \lwv s of weight $2e_m-e_1$ in $\Diff^{j-i-1}(\Symb^j, \Symb^i)$ explicitly: they are 
\begin{equation*}
   \mbox{\rm $\xi_1 D_{\xi_m}^2 \circ \Div^{j-i-1}$ (for $i \ge 1$), and 
   $D_1 D_{\xi_m}^2 \circ \Div^{j-i-2}$ (for $j-i\ge 2$),}
\end{equation*}
where $\Div$ denotes the divergence $\sum_i D_i D_{\xi_i}$, the unique $\db_m$-invariant operator from $\Symb^k$ to $\Symb^{k-1}$ for all $k$.

The situation is as follows: there are $\da_m$-relative $\Diff(\Symb^j, \Symb^i)$-valued 1-cochains $\beta_{ij}$ and $\gamma_{ij}$ of $\VRm$, defined by
\begin{equation} \label{betagamma}
   \beta_{ij}(x_m^2 D_1) := \xi_1 D_{\xi_m}^2 \circ \Div^{j-i-1}, \quad
   \gamma_{ij}(x_m^2 D_1) := D_1 D_{\xi_m}^2 \circ \Div^{j-i-2}
\end{equation}
(where $\beta_{ij} = 0$ for $i=0$ and $\gamma_{ij} = 0$ for $j-i=1$).  These 1-cochains span the space of all $\da_m$-relative 1-cochains, and so there are scalars $B^\lp_{ij}$ and $C^\lp_{ij}$ such that
\begin{equation*} 
   \pi^\lp_{ij} = B^\lp_{ij} \beta_{ij} + C^\lp_{ij} \gamma_{ij}.
\end{equation*}
This is the multidimensional analog of~(\ref{piijbeta}).  The main difference is that here the space of $\da_m$-relative 1-cochains is in general 2- rather than 1-dimensional.

Explicit formulas for $B^\lp_{ij}$ and $C^\lp_{ij}$ would constitute a complete solution to Problem~1.  To date they are known only for $p=0$ and $j-i=1$ or~$2$ \cite{LO99}.

\subsection*{Problem 3}
The first step is to seek some analog of Lemma~\ref{FCL} (we know of no such result, but see \cite{Le00}).  Next one should examine the cohomology of the \tfm s.  For example, one finds the following analog of~(\ref{VR1coho}):

\begin{prop}
The $\du_m$-relative $\db_m$-covariant 1-cohomology of\/ $\VRm$ with values in\/ $F\bigl(L(\lambda)\bigr)$  is 1-dimensional for $\lambda = 0$, $e_m$, or $2e_m-e_1$, and 0-dimensional for all other $\lambda$.  

Regarding\/ $\VRm$ as\/ $\Symb^1(0,0)$, the cocycle at $\lambda = 0$ is the divergence\/ $\Div$ (in particular, it is~$m$ on the Euler operator $E$).  The cocycle at $\lambda = e_m$ is $\db_m$-relative, and that at $\lambda = 2e_m-e_1$ is $\da_m$-relative.
\end{prop}

Concerning the cohomology of the modules of \dog s, the only work we know is \cite{LO00}, in which the 1-cohomology of $\Diff\bigl(\Symb^j(0), \Symb^i(0) \bigr)$ is computed: it is 1-dimensional for $j-i = 0$, $1$, or~$2$, and zero dimensional otherwise.  The symbol of the cocycle at $j-i=0$ is the $F\bigl(L(0)\bigr)$-valued cocycle $\Div$ above, and the symbols of those at $j-i=1$ and~$2$ are the $\da_m$-relative $F\bigl(L(2e_m-e_1)\bigr)$-valued cocycle.  More precisely, the cocycle at $j-i=1$ is $\beta_{j-1,j}$, and that at $j-i=2$ is a linear combination of $\beta_{j-2,j}$ and $\gamma_{j-2,j}$ (see~(\ref{betagamma})).

The 1-cohomology of $\Diff\bigl(\Symb^j(p), \Symb^i(p) \bigr)$ should be studied for arbitrary~$p$.  In the resonant cases there will a cocycle at $j-i=1$ whose symbol is the $\db_m$-relative $F\bigl(L(e_m)\bigr)$-valued cocycle.  This cocycle will be linked to $\beta_{j-1,j}$ via the coboundary of the {\em multidimensional affine Bol operator,\/} $\Div^{j-i}$ (see the discussion of the 1-dimensional resonant case).

Multidimensional extensions of the results of \cite{Go73} and \cite{FF80} to higher cohomology would be very interesting (and probably very difficult).

\subsection*{Problem 4}
The equivalence classes of the \sq s $\SQ^k_l(\lp)$ are known only for $p=0$ \cite{DO97, LMT96, LO99}.  The generalization to arbitrary~$p$ will be a corollary of the solution of Problem~1: the scalars $B^\lp_{ij}$ and $C^\lp_{ij}$ with $j-i\le 3$ will be the ingredients.  In the \dog\ setting, $p$, $k$, and~$l$ are all invariants, because the composition series is $\{\Symb^{k-j}(p): 0\le j< l\}$.  In the \psdog\ setting both $\lambda$ and $p$ can vary, and it appears likely that the equivalence classes have a rich structure.

\subsubsection*{Length~2}
Since $B^\lp_{i,0} = 0$, $\Diff^k$ (which is $\SQ^k_{k+1}$) and those $\SQ^k_l$ with $l\le k$ behave differently.  For example, $\SQ^k_2$ is split \iff\ $B^\lp_{k-1,k} = 0$.  This occurs only for exceptional (\eg\ self-conjugate) $\lambda$ if $k\ge 2$, but $\Diff^1$ is always split.

\subsubsection*{Length~3}
For $\Diff^2(\lp)$ there are \`a priori four equivalence classes, determined by whether each of $B^\lp_{1,2}$ and $C^\lp_{0,2}$ is zero or not.  By \cite{DO97}, when $p=0$ only three of these classes actually arise: the totally split class is missing.

For $k\ge 3$, the scalars relevant to $\SQ^k_3$ are $C^\lp_{k-2,k}$ and $B^\lp_{k-i,k-j}$ for $(i,j)$ equal to any of $(1,0)$, $(2,1)$, and $(2,0)$.  By Problem~3, the dimension of the space of 1-cocycles corresponding to each of the entries on the first two superdiagonals is one.  It follows that in addition to three discrete (two-valued: zero or non-zero) invariants there is one continuous one, as in the lacunary \sq\ case for $\VR$.  This means that the number of equivalent \sq s with a given $p$ is the $\lambda$-degree of the invariant at~$p$ (which is unknown).  For example, at $p=0$ each \sq\ is equivalent only to its conjugate \cite{LO99}.  However, in the \psdog\ case the equivalence classes will be curves in $(\lp)$-space.

\subsubsection*{Length~$\ge 4$}
$\Diff^3(\lp)$ has four discrete (zero/non-zero) invariants, corresponding to $B^\lp_{1,2}$, $B^\lp_{2,3}$, $C^\lp_{0,2}$, and a linear combination of $B^\lp_{1,3}$ and $C^\lp_{1,3}$ ($C^\lp_{0,3}$ and the other scalar in the $(1,3)$ spot do not give invariants, because they are not coefficients of cocycles).  It also has one continuous invariant.  Thus we expect finite equivalence classes for fixed $p$, but curves in the \psdog\ case.

For $k\ge 4$, $\SQ^k_4(\lp)$ has two continuous invariants, so even in the \psdog\ case we expect only finite equivalence classes.  In higher length we expect that each \sq\ is equivalent only to its conjugate.

\subsection*{Problem 5}
Consider \ind\ $\da_m$-split extensions of $\Symb^j(p)$ by $\Symb^i(p)$.  By \cite{LO00} (see Problem~3), for $p=0$ there is one such at $j-i=1$ and~$2$ and none otherwise.  It is possible that for exceptional~$p$, there is one at $j-i=3$ (the solution of Problem~1 would tell).  The $\da_m$-split extensions of arbitrary pairs of \tfm s should be classified and realized naturally, possibly as \sq s of \psdog\ modules.

\section{Contact structures and conformal quantizations: Con$\th\bR^{2\ell+1}$}
In this section let $m = 2\ell+1$ be odd, fix coordinates $\{x_i, y_i, z: 1\le i\le \ell\}$ on $\bR^m$, and write $\Poly\bR^m$ for $\bC[x_1, \ldots, y_1, \ldots, z]$.  The standard {\em contact form\/} on $\bR^m$ is $\omega_m := dz + \frac{1}{2} \sum_1^\ell (x_i dy_i - y_i dx_i)$.  Note that $\omega_m\wedge (d\omega_m)^\ell \not=0$.

The associated completely nonintegrable distribution is $\Tan\bR^m$, the space of vector fields tangent to (\ie\ annihilated by) $\omega_m$.  The {\em contact Lie algebra\/} is
$$ \Con\bR^m\ :=\ \Stabilizer_{\VRm} \bigl(\Tan\bR^m\bigr)\ =\ 
   \Stabilizer_{\VRm} \bigl(\omega_m \Poly\bR^m\bigr). $$

It is a fact that $\VRm = \Con\bR^m \oplus \Tan\bR^m$.  However, although $\Tan\bR^m$ is a $\Poly\bR^m$-module under multiplication, $\Con\bR^m$ is not.  Recall the \tdm s $F(\lambda)$ of $\VRm$.  Their restrictions to $\Con\bR^m$ are still \irr\ (excepting $F(0)$, which remains of length~2), and there is a $\Con\bR^m$-equivalence
$$ \chi: F\bigl(-\ts\frac{1}{\ell+1}\bigr) \to \Con\bR^m, \quad
   \mbox{\rm defined by} \quad
   \bigl\la \omega_m, \chi\bigl((dx dy dz)^{-1/(\ell+1)}f\bigr) \bigr\ra := f. $$

For $m,\th k\ge 1$, the symbol modules $\Symb^k(p)$ of $\Diff(\lp)$ are not \irr\ under $\Con\bR^m$.  As implied by \cite{Ov06}, there is a $\Con\bR^m$-invariant {\em fine filtration\/} $\bigl\{\Diff^{k,l}(\lp): 0\le k,\ 0\le l\le k\bigr\}$ of $\Diff(\lp)$:
\begin{equation*}
   \Diff^{k,l}(\lp) := \Diff^{k-1}(\lp) +
   \sum_{i=0}^l  \bigl(\Diff^i(\lp)\bigr) \bigl(\Tan\bR^m\bigr)^{k-i}.
\end{equation*}
Geometrically, $\Diff^{k,l}(\lp)$ consists of the \dog s of order~$\le k$ and non-tangential order~$\le l$.  Let $\Symbf^{k,l}(p)$ be the corresponding {\em fine symbol modules\/} $\Diff^{k,l}/\Diff^{k,l-1}$, and write $\sigma^p_{k,l}$ for the action of $\Con\bR^m$ on them.  They are generally \irr\ under $\Con\bR^m$.  Define $\Symbf(p) := \bigoplus_{k,l} \Symbf^{k,l}(p)$.

$\Symb^k(p)$ inherits its own fine filtration, $\Symb^{k,l}(p) := \Diff^{k,l}/\Diff^{k-1}$.  Thus
$$ \Symbf^{k,l}(p) = \Symb^{k,l}(p) / \Symb^{k,l-1}(p). $$
In general this filtration does not split under $\Con\bR^m$ \cite{FMP07}.

The finite dimensional simple maximal subalgebra $\da(\Con\bR^m)$ of $\Con\bR^m$ is the {\em conformal subalgebra.\/}  It is $\da(\VRm) \cap \Con\bR^m$, the image of the polynomials of degree~$\le 2$ under $\chi$, and is isomorphic to $\dsp_{m+1}$.  We denote it by $\ds_m$:
$$ \ds_m := \chi\bigl((dx dy dz)^{-1/(\ell+1)} \bigl\{f\in\Poly\bR^m:
   \degree(f) \le 2 \bigr\} \bigr) \cong \dsp_{m+1}. $$

A {\em conformal quantization\/} of $\Diff(\lp)$ is a symbol-preserving $\ds_m$-equivalence from $\Symbf(p)$ to $\Diff(\lp)$.  One of the main results of \cite{FMP07} is that for most $p$ ($p\not\in -\frac{1}{m+1}\bN$ is sufficient), the fine filtration of the full symbol module $\Symb^k(p)$ splits under $\ds_m$.  Coupling this with Proposition~\ref{tdmPQ} gives:

\begin{prop} \label{tdmCQ}
For $p\not\in \frac{1}{m+1}\bZ$, there is a unique conformal quantization
$$ \CQ_\lp: \Symbf(p) \to \Diff(\lp). $$
\end{prop}

Consider the five problems of Section~\ref{PQs} with $\Con\bR^m$ replacing $\Vec\bR^m$ and $\CQ$ replacing $\PQ$.  Problem~1 is largely open and the rest are completely open.  Regarding Problem~3, we expect that as usual, most if not all 1-cocycles appear as matrix entries on the first few superdiagonals in Problem~1.

\subsection*{Problem~1}
Here $\pi^\lp$ is the \r\ $\CQ^{-1}_\lp \circ \sigma^\lp \circ \CQ_\lp$ of $\Con\bR^m$ on $\Symbf(p)$.  Its ``matrix entries'' are the constituent maps
$$ \pi^\lp_{ij;\th ab}: \Con\bR^m \to \Hom_\bC\bigl(\Symbf^{j,b}(p),\ \Symbf^{i,a}(p)\bigr). $$
The analog of Lemma~\ref{tdmpiij} is immediate:

\begin{lemma} \label{tdmpiijab}
\begin{enumerate}
\item[(a)]  $\pi^\lp_{ij;\th ab} = 0$ for $i>j$, and also for $i=j$ and $a>b$.
\item[(b)]  $\pi^\lp_{ii;\th aa} = \sigma^p_{i,a}$, the action of\/ $\Con\bR^m$ on\/ $\Symb_\fine^{i,a}(p)$.
\item[(c)]  $\pi^\lp_{ij;\th ab}$ is an $\ds_m$-relative 1-cochain for $i<j$, and also for $i=j$ and $a<b$.
\end{enumerate}
\end{lemma}

The progress toward computing the $\pi^\lp_{ij;\th ab}$ so far consists mainly in proving that certain of them vanish.  There are essentially three results.  First, it is observed in \cite{FMP07} that contraction with the contact form $\omega_m$ is a $\Con\bR^m$-surjection from $\Symb^k(p)$ to $\Symb^{k-1}(p-\frac{1}{\ell+1})$ with kernel $\Symbf^{k,0}(p)$.  It follows that for $a\ge 1$, $\pi^\lp_{kk;\th ab}$ is (when viewed properly) equal to $\pi^{\lambda,\th p-1/(\ell+1)}_{k-1,k-1;\th a-1,b-1}$.  (Keep in mind that $\Symb^k(p)$ and hence also $\pi^\lp_{kk;\th ab}$ are independent of $\lambda$.)

It is easy to verify $\pi^\lp_{11;\th 01}=0$, whence the above observation yields 
$$ \pi^\lp_{kk;\th k-1,k}=0 \quad \mbox{\rm for $k\ge 1$.} $$

The other two results concern entries $\pi^\lp_{ij;\th ab}$ with $i<j$, and will be published in \cite{CO}.  So far they have been verified only for $\Con\bR^3$.  The first of them is
\begin{equation*}
   \pi^\lp_{k-1,k;\th k-1,b} = 0 \quad \mbox{\rm for $k\ge 1$ and $0\le b\le k$.}
\end{equation*}
In other words, at $m=3$, $\Diff^k / \Diff^{k-1, k-2}$ splits as $\Symb^k \oplus \Symbf^{k-1, k-1}$ under $\Con\bR^3$.  Put geometrically, this says that there is a $\Con\bR^3$-invariant purely contact subsymbol.  In light of the fact that at $m=1$, $\VR$ and $\Con\bR$ are equal and $\Tan\bR=0$, this generalizes the $\VR$-invariant subsymbol discussed in Section~\ref{R} (see the length~2 case of Problem~4 there).

The last result (again, only verified for $\Con\bR^3$) is as follows: 
\begin{eqnarray*}
   \pi^\lp_{k-1,k;\th ab} &=& 0 \quad \mbox{\rm for $k\ge 1$ and $a\ge b+1$,} \\
   \pi^\lp_{ij;\th ab} &=& 0 \quad \mbox{\rm for $i\le j-2$ and $a\ge b+2$.}
\end{eqnarray*}
This has the following curious consequence: $\Diff$ has a $\Con\bR^3$-invariant filtration $\Diff^{(K)}$, defined by (for $l\ge k$, $\Symb^{k,l}$ means $\Symb^k$)
$$ \Diff^{(K)}(\lp) = \CQ_\lp\Bigl( \bigoplus_{i=0}^K \Symb^{K-i, \lf i/2 \rf}(p) \Bigr) $$
Intersecting this with the usual order filtration gives a bi-filtration of $\Diff$.  

Observe that the \sq\ $\Diff^{(K)}/ \Diff^{(K-1)}$ is a $\Con\bR^3$-module with composition series $\Symbf^{K,0},\ \Symbf^{K-2,1},\ \Symbf^{K-4,2}, \ldots$.  It might be interesting to classify the equivalence classes of the low-length truncations of these modules.

\def\pxi{D_\xi}
\def\px{D_x}

\section{Con$\th\bR^{1|1}$}
We conclude with a discussion of the contact structure of the superline $\Roo$.  The five problems of Section~\ref{PQs} make sense for $\Vec\bR^{m|n}$ and $\Con\bR^{2\ell+1 |n}$, and the methods and results for $\Con\Roo$ are closely parallel to those for $\VR$.

Let $x$ and $\xi$ be the even and odd coordinates, respectively.  Set $D := \pxi + \xi\px$ and $\o D := \pxi - \xi\px$.  (N.B.: $D^2 = \px = -\o D^2$.)  The contact form $\omega$ is $dx + \xi d\xi$, and its nonintegrable distribution is $\Tan\Roo := (\Poly\Roo)\o D$.  The conformal \lsa\ $\Con\Roo$ is the $\VRoo$-stabilizer of $\Tan\Roo$.  It is generated by its odd part, $(\Con\Roo)_\odd = \bC[x] D$.  As before, $\VRoo = \Tan\Roo \oplus \Con\Roo$.

$\Con\Roo$ is also the stabilizer of $(\Poly\Roo)\omega$, which is the space of sections of a line bundle.  We define the \tdm\ $F(\lambda)$ to be $(\Poly\Roo)\omega^\lambda$, the $\lambda^\thup$ scalar power of this bundle.  There is a $\Con\Roo$-equivalence $\chi: F(-1) \to \Con\Roo$, defined on odd elements by $\omega^{-1}\xi g(x) \mapsto \frac{1}{2} gD$.

The space $\Diff(\lp)$ of \dog s from $F(\lambda)$ to $F(\lambda+p)$ is spanned over $\Poly\Roo$ by $\omega^p \{ \o D^i: i\in\bN\}$.  One checks that $g(x) D$ acts on $F(\lambda)$ by the \dog\ $gD + 2 \lambda \xi g'$.  In this context, the {\em fine filtration\/} was introduced in \cite{GMO07}.  It is the $\Con\Roo$-invariant $\bN/2$-filtration
$$ \Diff^k(\lp) := \Span_{\Poly\Roo} \omega^p \bigl\{\o D^i: 0\le i\le 2k \bigr\},
   \quad k\in\bN/2. $$
Its symbol modules $\Symb^k_\fine(p) := \Diff^k(\lp) / \Diff^{k-1/2}$ are $\Con\Roo$-equivalent to $F(p-k)^{2k\Pi}$, where $\Pi$ is the parity functor.  Set $\Symb_\fine(p) := \bigoplus_{\bN/2} \Symb^k_\fine(p)$.

The conformal subalgebra $\ds_\oo$ is again $\chi$ applied to the degree~$\le 2$ polynomials.  It is isomorphic to $\osp_{1|2}$.  A {\em conformal quantization\/} is an even symbol-preserving $\ds_\oo$-equivalence from $\Symb_\fine$ to $\Diff$.  A Casimir operator argument gives:

\begin{prop}
For $p\not\in \frac{1}{2}\bZ^+$, there is a unique conformal quantization
$$ \CQ_\lp: \Symb_\fine(p) \to \Diff(\lp). $$
\end{prop}

The explicit formula for $\CQ_\lp$ was deduced in \cite{CMZ97} for $(\lambda,p)=(0,0)$, and in
\cite{GMO07} in general.  Problems~1 and~5 are largely solved in \cite{Co08}, and Problem~4 was reduced to computation.

\subsection*{Problem 1}
The matrix entries here are even $\ds_\oo$-maps
$$ \pi^\lp_{ij}: \Con\Roo \to \Hom\bigl(F(p-j)^{2j\Pi}, F(p-i)^{2i\Pi}\bigr),
   \quad i,\th j\in \bN/2. $$
As usual, the matrix is upper triangular with the tensor density actions on the diagonal.  The entries above the diagonal are $\ds_\oo$-relative 1-cochains.

Now $\Con\Roo / \ds_\oo$ is $\ds_\oo$-equivalent to $F(3/2)^\Pi$, with \lwv\ $\xi x^2 D$.  It follows that the space of $\Hom\bigl(F(\mu), F(\mu+q)\bigr)$-valued $\ds_\oo$-relative 1-cochains of $\Con\Roo$ is zero unless $q\in \frac{3}{2} + \frac{1}{2}\bN$, when it is 1-dimensional.  

Let us describe this space.  For $\mu+\nu\not\in -\bN/2$ and $k\in \bN/2$, there is up to a scalar a unique $\ds_\oo$-surjection from $F(\mu) \ot F(\nu)$ to $F(\mu+\nu+k)$, the {\em supertransvectant\/} $J^{\mu,\nu}_k$, which has parity $(-1)^{2k}$.  For $q\in \frac{1}{2} + \bN$ and $\mu = \frac{1}{4}-\frac{q}{2}$, there is up to a scalar a unique $\ds_\oo$-map from $F(\mu)$ to $F(\mu+q)$, the {\em super Bol operator\/} $\SBol_q := \omega^q \o D^{2q}$, which is odd.  The above space of $\ds_\oo$-relative 1-cochains is spanned by 
$$ \o\beta_q(\mu) := J^{3/2,\mu}_{q-3/2} \circ \bigl((\SBol_{5/2}\circ\chi) \ot 1\bigr). $$

In order to compute the matrix entries, it is necessary to consider the spaces of even $\ds_\oo$-relative $\Hom\bigl(F(\mu), F(\mu+q)^{2q\Pi}\bigr)$-valued 1-cochains with $q\in\bZ^+/2$.  This space is zero when $q$ is $\frac{1}{2}$ or~$1$.  Otherwise it is 1-dimensional and spanned by a cochain $\beta_q(\mu)$, which is $\o\beta_q(\mu)$ modified by an appropriate parity equivalence.  One checks easily that there are scalars $B^\lp_{ij}$ such that $\pi^\lp_{ij} = B^\lp_{ij} \beta_{j-i}(p-j)$.

These scalars may be computed by the method of Section~\ref{R}.  This is carried out in \cite{Co08}, although the formulas are fully simplified only when $q$ is $\frac{3}{2}$, $2$, and $\frac{5}{2}$.  This is sufficient, as these are the only $q$ for which $\beta_q(\mu)$ is generically a cocycle.

\subsection*{Problem 3}
Only $H^0$ and $H^1$ are fully known.  $H^1 F(\lambda)$ is zero except for:
$$ H^1 F(0) = \bC \omega \o D^2 \circ \chi, \quad
   H^1 F(1/2) = \bC \omega \o D^3 \circ \chi, \quad
   H^1 F(3/2) = \bC \omega \o D^5 \circ \chi. $$
The first of these is $D$-relative and affine-covariant, the second is affine-relative, and the third is $\ds_\oo$-relative.  Based on preliminary calculations for $n\le 3$, it is conjectured in \cite{Co08} that the affine-relative $n$-cohomology of $F(\lambda)$ is 1-dimensional when $\lambda$ is $n^2 \pm n/2$, and zero otherwise.

$H^1 \Diff(\lp)$ is computed in \cite{Co08}.  We will only give the affine-relative cohomology: it is 1-dimensional if $p=\frac{1}{2}$ and $\lambda=0$; $p=\frac{3}{2}$, $2$, or $\frac{5}{2}$; $p=3$ and $\lambda=-\frac{5}{2}$ or~$0$; or $p=4$ and $\lambda = -\frac{7}{4} \pm \frac{1}{4}\sqrt{33}$ (the $p=4$ cocycles are the analogs of the Feigin-Fuchs 1-cocycles of $\VR$).  Otherwise it is zero.  It is $\ds_\oo$-relative and spanned by $\o\beta_p(\lambda)$ whenever $C := (\lambda + \frac{p}{2} - \frac{1}{4})^2$ is nonzero.

Some information concerning $H^2 F(\lambda)$ and $H^2 \Diff(\lp)$ is given in \cite{Co08}.  It seems natural to guess that the picture for $H^n$ is broadly as in \cite{Go73} and \cite{FF80}; precise results would be highly interesting.

\subsection*{Problem 4}
Let $\SQ^k_l(\lp)$ be as in~(\ref{PDOSQ}).  Here $k, l\in\bN/2$ and the \sq\ is of length $2l$.  We allow \psdog s, so $l$ and $p-k$ are invariants and $\lambda$ and $p$ vary continuously.  The classification of the equivalence classes was outlined in \cite{Co08}; details will be forthcoming.  The most interesting cases occur in lengths~6 and~7, where there are continuous invariants.  There are also interesting lacunary \sq s.

\subsubsection*{Length 6}
One finds the usual discrete invariants, and a single continuous invariant like the first one in~(\ref{invts}).  In $(C,p)$ coordinates, the generic equivalence classes make up a pencil of conics through four fixed points.   The coordinate system in which these four points are inscribed in a circle has a form similar to that of its $\VR$-analog, and merits further study.

\subsubsection*{Length 7}
Since $\pi^\lp_{k-7/2,\th k}$ is not a cocycle, two \sq s are equivalent \iff\ their two pairs of length~6 \sq s are equivalent.  Hence there may be interesting finite equivalence classes here: the intersections of conics.

\subsection*{Problem 5}
The indecomposable modules composed of two \tdm s are classified by Problem~3.  In \cite{Co08}, several length~3 and~4 uniserial \sq s of $\Psi(\lp)$ are computed.  Exceptional tensor degrees arise in length~4.

\def\eightit{\it} 
\def\bib{\bf}
\bibliographystyle{amsalpha}

\end{document}